\documentclass[12pt,a4paper]{amsart}
\usepackage{amssymb,amsfonts}
\usepackage[all,arc]{xy}
\usepackage{enumerate}
\usepackage{mathrsfs}
\usepackage{systeme}
\usepackage{tikz}
\usepackage{pgfplots}
    \pgfplotsset{
        compat=1.3,
    }
\usepackage{graphicx}
\usepackage{graphics}
\usepackage{tkz-euclide}
\usepackage{pgfplots}
    \usetikzlibrary{calc}
    \pgfplotsset{compat=1.12}
    
\usetikzlibrary{intersections}

\newtheorem*{prob*}{Problem}
\newtheorem{thm}{Theorem}[section]
\newtheorem{cor}[thm]{Corollary}
\newtheorem{prop}[thm]{Proposition}
\newtheorem{lem}[thm]{Lemma}

\newtheorem{quest}[thm]{Question}

\theoremstyle{definition}

\theoremstyle{remark}
\newtheorem{rem}[thm]{Remark}

\makeatletter
\let\c@equation\c@thm
\makeatother

\newcommand{\Ra}{\mathcal{R}_a}

\begin{document}

\title{A Congruence Condition For The Four-Distance Problem}

\author{William McCloskey}

\begin{abstract}  Place the vertices of a rectangle at $\{(0, \pm 1/2), (a, \pm 1/2)\}$, where $a$ is rational. Any point $(x, y)$ that is rational distance from all four vertices of the rectangle must have rational coordinates. We prove that if $v_3(a) = 0,$ then either $v_3(x)<0$ or $v_3(y)<0$, where $v_3(\cdot)$ is the 3-adic valuation. The case of particular interest is the long-open four-distance problem, which asks whether such a rational distance point exists in the case $a=1$ of the unit square. For the four-distance problem, our result rules out one-fourth of all potential solutions with bounded height.
\end{abstract}

\maketitle

The four-distance problem is a long-open problem which asks whether there is a point in the plane at rational distance from the four corners of the unit square. It is believed that no such point exists. The problem and much of the relevant work is covered in detail in Section D19 of Richard Guy's \textit{Unsolved Problems in Number Theory}. Previous approaches to the four-distance problem include tiling the square with rational triangles (\cite{BremnerGuyLambda}, \cite{BremnerGuyNu}, \cite{GuyTiling}) and splitting the problem into two equations -- one involving distances to three corners of the square and the other involving the fourth distance (\cite{Berry}, \cite{berry_1992}). 

In this article, we split problem by considering separately two distances from one side of the square and two distances from the other side of the square. Specifically, we say that the two-distance problem is to find a point $(x,y)$ rational distance from both points $(0, \pm1/2)$. For the four-distance problem, we let $\Ra$ denote the vertices of the rectangle $(0,\pm 1/2), (a, \pm 1/2)$. We say $(x,y)$ is rational distance from the rectangle $\Ra$ if $(x,y)$ is rational distance from all points $(0,\pm 1/2), (a, \pm 1/2)$, and the four-distance problem for $\Ra$ is to find a point $(x,y)$ rational distance from $\Ra.$ This formulation is useful because the four-distance problem for the rectangle $\Ra$ is equivalent to finding solutions $(x,y),(x-a,y)$ to the two-distance problem.

In the first section, we prove some preliminaries for the two-distance problem. We present a polynomial $p_{x,z}(u)$ and prove that it has the following property: given rational numbers $x$ and $z$, the polynomial $p_{x,z}(u)$ has a nonzero rational root $u$ if and only if $(x,z/2)$ is a solution to the two-distance problem. Thus the four-distance problem for $\Ra$ is solvable if and only if $p_{x,z}(u)=0, p_{x-a,z}(\mu)=0$ has a rational solution $x,z,u,\mu$ with $u,\mu$ nonzero. 

In the second section, we consider the four-distance problem for the rectangle $\Ra$. We prove that if $v_3(a)=0$ then any point $(x,y)$ rational distance from $\Ra$ satisfies $v_3(x)<0$ or $v_3(y)<0$. (Any solution to the four-distance problem must have rational coordinates.) In particular, our result holds in the case of interest $a=1$, the unit square. Our result rules out one-fourth of the points that would be considered in a computer search of rational points with bounded height.

In the third section, we discuss possibilities of generalizing the argument for the cases where $v_3(x)<0$ or $v_3(y)<0$. Continuing to work with the system $p_{x,z}(u)=0, p_{x-a,z}(\mu)=0$ looks difficult, but there is some hope after translating the results to the equations for the three-distance problem and fourth distance studied by other papers.

\section{The Four-Distance Problem: Approach And Preliminaries}

\bigskip
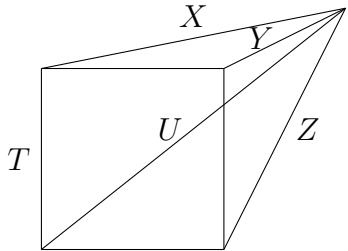
\begin{center}
\begin{tikzpicture}[scale=.8]
\draw (0,0) -- (0,3) node[midway, left]{$T$};
\draw (0,3) -- (3,3);
\draw (3,3) -- (3,0);
\draw (0,0) -- (3,0);
\draw (0,0) -- (5,4) node[midway,left]{$U$};
\draw (3,0) -- (5,4) node[midway,right]{$Z$};
\draw (0,3) -- (5,4) node[midway,above]{$X$};
\draw (3,3) -- (5,4) node[midway,left]{$Y$};
\end{tikzpicture}
\end{center}

One of the main approaches to the four-distance problem is given by the following set of equations
\begin{align}
2(Y^4+T^4)+X^4+Z^4&=2(X^2+Z^2)(Y^2+T^2)  \\
U^2+Y^2&= X^2+Z^2.
\end{align} (See \cite{Berry} or \cite{GuyUnsolvedProblems}.) Equation (1) is equivalent to the three-distance problem: finding rational distances $X,Y,Z$ to the corners of a square with rational side length $T$. Given a rational solution to (1), equation (2) determines whether the fourth distance $U$ is also rational. For some time, it was believed there does not exist a solution to the three-distance problem that is not on a side of the square. Thus, these equations are quite logical: if the three-distance problem is already difficult, it makes sense to separate the four-distance problem into a three-distance problem and a one-distance problem. 

Moreover, substantial progress has been made in understanding the three-distance problem. In 1967, a one-parameter family of nontrivial solutions to the three-distance problem was discovered by J.H. Hunter. This family of solutions was rediscovered later by John Leech, John Conway, and Mike Guy, and Leech showed how to construct infinitely many one-parameter families from this one. In \cite{Berry}, T.G. Berry used the geometry of the surface (1) to generate even more one-parameter families from given ones, and these included the Hunter-Leech-Conway-Guy solutions as well as the families constructed by Leech.

In this paper, rather than narrowing down the problem with the three-distance problem and then checking the fourth distance, we look at a two-distance problem for the left side of the square and a two-distance problem for the right side of the square. We also work in the coordinate plane.

\begin{prob*}[Four-Distance Problem For A Rectangle] \label{fdp}
The four-distance problem for the rectangle of side lengths $1,a$ is to find a point in the plane rational distance from each corner of the rectangle. Place the vertices of the rectangle at $(0, \pm1/2), (a, \pm1/2).$ Then the problem is equivalent to finding a point $(x,y)$ that is rational distance from $(0, \pm 1/2)$ such that $(x-a,y)$ is also rational distance from $(0,\pm1/2).$ The four-distance problem is the case a=1.
\end{prob*}

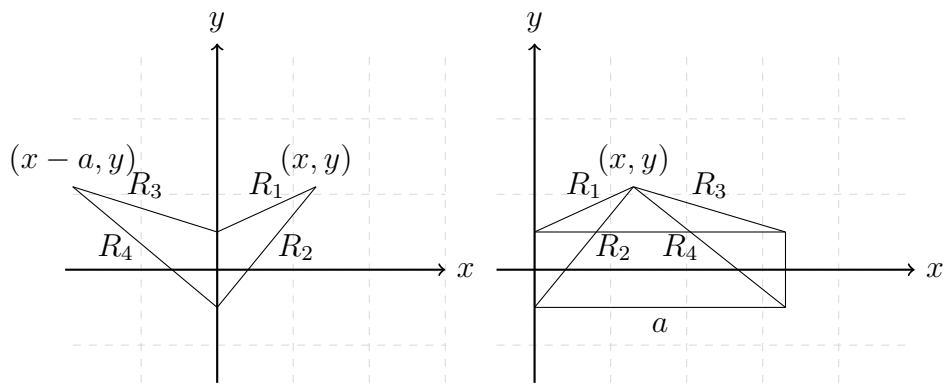
\begin{figure}[!h]  
\centering 

\begin{tikzpicture}
\draw[help lines, color=gray!30, dashed] (-1.9,-1.5) grid (3,2.9);
\draw[->, thick] (-2,0)--(3,0) node[right]{$x$};
\draw[->, thick] (0,-1.5)--(0,3) node[above]{$y$};

\draw (0,.5)--(1.3,1.1) node[midway,above]{$R_1$} node[above]{$(x,y)$};
\draw (0,-.5)--(1.3,1.1) node[midway,right]{$R_2$};

\draw (0,.5)--(-1.9,1.1) node[midway,above]{$R_3$} node[above]{$(x-a,y)$};
\draw (0,-.5)--(-1.9,1.1) node[midway,left]{$R_4$};

\end{tikzpicture}
\begin{tikzpicture}
\draw[help lines, color=gray!30, dashed] (-.5,-1.5) grid (4.9,2.9);
\draw[->, thick] (-.5,0)--(5,0) node[right]{$x$};
\draw[->, thick] (0,-1.5)--(0,3) node[above]{$y$};
\draw (0,.5)--(1.3,1.1) node[midway,above]{$R_1$} node[above]{$(x,y)$};
\draw (0,-.5)--(1.3,1.1) node[midway,right]{$R_2$};
\draw (3.3,.5)--(1.3,1.1) node[midway,above]{$R_3$};
\draw (3.3,-.5)--(1.3,1.1) node[midway,left]{$R_4$};
\draw (0,.5)--(0,-.5);
\draw (0,.5)--(3.3,.5);
\draw (3.3,.5)--(3.3,-.5); 
\draw (3.3,-.5)--(0,-.5)  node[midway, below]{$a$};
\end{tikzpicture}
 \label{fig:L}  
\end{figure}  

As in the introduction, we will denote the set $\{(0, \pm 1/2), (a, \pm 1/2)\}$ by $\Ra$.\footnote{This is the same notation as in Bremner and Ulas's paper \cite{BremnerUlas}, except our rectangle is translated downwards by $1/2$.} And we will say that $(x,y)$ is rational distance from $\Ra$ if $(x,y)$ is rational distance from all of the elements of $\Ra$.

For our formulation of the four-distance problem for a rectangle, we want to study the points $(x,y)$ that are rational distance from $(0, \pm1/2).$ We will refer to this as the two-distance problem.

\begin{prob*}[Two-Distance Problem] The two-distance problem is to find a point $(x,y)$ rational distance from $(0, \pm 1/2).$

\end{prob*} Solutions to the two-distance problem have the following restriction.

\begin{lem} \label{lem:yrational}

If $(x, y)$ is a solution to the two-distance problem, then $y$ is rational. \end{lem}

\begin{proof}
By assumption, for some rational numbers $R_1$ and $R_2$, we have\[x^2 + (y - 1/2)^2 = R_1^2\]
\[x^2 + (y + 1/2)^2 = R_2^2.\] Subtracting the top equation from the bottom one, we see that
\[y = (R_2^2 - R_1^2)/2,\] so $y$ is rational.

\end{proof}

\begin{cor} \label{cor:ratcoords}
Suppose that $a$ is rational. If $(x,y)$ is rational distance from $\Ra$, then $x$ and $y$ are rational.
\end{cor}

Because of the preceding corollary,  we are mainly interested in studying points $(x,y)$ with rational coordinates. To do so, we will make great use of the polynomial \[p_{x,z}(u) := u^4-(z^2+4x^2+1)u^2+z^2.\]

\begin{thm} \label{phzu}
Let P = $(x, z/2)$ be a fixed point in the plane. Let $R_1$ be the distance of $P$ to $(0, 1/2)$ and let $R_2$ be the distance of $P$ to $(0, -1/2)$. Then the four roots of the quartic $p_{x,z}(u)$ are $R_1\pm R_2$ and $-(R_1 \pm R_2)$.
\end{thm}
\begin{proof}
First, let us check that $R_1\pm R_2$ and $-(R_1 \pm R_2)$ are indeed roots. Since we are squaring $u$, we only need to plug in $u=R_1+\epsilon R_2$, where $\epsilon=\pm1$. 

Writing $y=z/2$, the coefficient \begin{align*}
    z^2+4x^2+1 &= 4y^2+4(R_2^2-(y+1/2)^2)+1 \\
    &=4(R_2^2-y).
\end{align*} In the proof of Lemma \ref{lem:yrational}, we showed $y=(R_2^2-R_1^2)/2$. So we find
\[z^2+4x^2+1=2(R_1^2+R_2^2).\] Thus we get $p_{x,z}(R_1+\epsilon R_2$) is \[(R_1+\epsilon R_2)^4-2(R_1^2+R_2^2)(R_1+\epsilon R_2)^2+(R_2^2-R_1^2)^2,\] which is 0.

All four roots are distinct except in the special  case $R_1 = R_2$, where $z/2 = 0$. Here the polynomial $p_{x,z}(u)$ is just $u^2(u^2-(4x^2+1))$, so $0=R_1 - R_2 = R_2-R_1$ is indeed a double root. 
\end{proof}

\begin{cor} \label{cor:ratdistpoly}
A point with rational coordinates $(x,z/2)$ is a solution to the two-distance problem if and only if $p_{x,z}(u)$ has a nonzero rational root.
\end{cor}
\begin{proof}
If $(x,z/2)$ is a solution to the two-distance problem, then $R_1+R_2$ is a nonzero rational root. For the converse, note that $R_1$ and $R_2$ are square roots of rational numbers because $x,z/2$ are rational. This means that if any of $R_1\pm R_2$ are nonzero and rational, then $R_1$ and $R_2$ are rational.
\end{proof}

Combining Corollary \ref{cor:ratcoords} and Corollary \ref{cor:ratdistpoly}, we get a reformulation of the four-distance problem.

\begin{prop} \label{prop:fdpsyst}
Suppose that $a$ is rational. A point $(x,z/2)$ is at rational distance from $\Ra$ if and only if $x, z$ are rational and $p_{x,z}(u)$ and $p_{x-a,z}(\mu)$ have rational roots that are nonzero. This means finding a solution $x,z,u,\mu$ in rationals ($u,\mu \neq 0)$ to the following system of equations
\begin{align*}
u^4-(z^2+4x^2+1)u^2+z^2&=0 \\
\mu^4-(z^2+4(x-a)^2+1)\mu^2+z^2&=0.
\end{align*}
The four-distance problem is the case $a=1$.
\end{prop}

A congruence condition for this system of equations is the main result of the next section.

\section{A Congruence Condition For The Four-Distance Problem}

In this section, we show that the system of equations in Proposition \ref{prop:fdpsyst} has no rational solutions with $v_3(x) \geq 0$ and $v_3(z) \geq 0$ if $v_3(a)=0$. Note that this holds in the case $a=1$, the four-distance problem. 

For a given prime $p$, the $p$-adic valuation $v_p$ is defined as usual: any nonzero rational number $t$ can be written uniquely as $t= p^k \frac{r}{s}$ where $r$ and $s$ are coprime integers not divisible by $p$, and we define $v_p(t) = k$. And $v_p(0)$ is defined to be $\infty$. The $p$-adic valuation satisfies the property that $v_p(s+t) \geq \min\{v_p(s),v_p(t)\}$, with equality unless $v_p(s)=v_p(t).$
\begin{prop} \label{prop:v3hv3z<0}
Suppose that $(x,z/2)$ is rational distance from $\Ra$, where $a$ is rational with $v_3(a) = 0$. Then $v_3(x)<0$ or $v_3(z)<0$.
\end{prop}

\begin{proof}
By Proposition \ref{prop:fdpsyst}, any point $(x,z/2)$ rational distance from $\Ra$ gives a solution in rationals ($u,\mu \neq 0$) to the system
\begin{align*}
u^4-(z^2+4x^2+1)u^2+z^2&=0 \\
\mu^4-(z^2+4(x-a)^2+1)\mu^2+z^2&=0.
\end{align*}
Proceeding by contradiction, suppose that $v_3(x)\geq 0$ and $v_3(z) \geq 0$. Then also $v_3(u)\geq 0$ and $v_3(\mu) \geq 0,$ so these equations can be taken (mod $3^k$). The first step is to show $z^2 \equiv 0 \pmod{3^k}$ for all $k$, which will imply that $z$ = 0.

We start by showing $z^2 \equiv 0 \pmod{3}$. Suppose on the contrary that $z^2 \equiv 1 \pmod{3}$. Since $u^4 \equiv u^2 \pmod{3}$ and $\mu^4 \equiv \mu^2 \pmod{3}$, we get the system \begin{eqnarray*}
-(1+4x^2)u^2+1&\equiv&0 \pmod{3}\\
-(1+4(x-a)^2)\mu^2+1&\equiv&0 \pmod{3}
\end{eqnarray*}

We see neither $u^2$ nor $\mu^2$ can be 0 (mod 3), so they must both be 1 (mod 3). We are left with the equations $-4x^2 \equiv 0 \pmod{3}$ and $-4(x-a)^2 \equiv 0 \pmod{3}$. By assumption $a$ is nonzero (mod 3), so these equations have no solution. Therefore we must have had $z^2 \equiv 0 \pmod{3}$.

For the induction step, suppose that $z^2 \equiv 0 \pmod{3^k}$. Then we get \begin{eqnarray*}
u^2(u^2-(4x^2+1))&\equiv&0 \pmod{3^k}\\
\mu^2(\mu^2-(4(x-a)^2+1))&\equiv&0 \pmod{3^k}.
\end{eqnarray*}
Notice that either $x \not\equiv 0 \pmod{3}$ or $x \not\equiv a \pmod{3}$. In the first case, $u^2 -(4x^2+1)$ is nonzero (mod 3), so we get $u^2\equiv0 \text{ (mod } 3^k).$ Similarly, in the second case we get $\mu^2 \equiv 0 \text{ (mod } 3^k).$ 

Consider the equations (mod $3^{k+1}$). We either have $z^2 \equiv 0 \pmod{3^{k+1}}$ or $z^2 \equiv 3^k \pmod{3^{k+1}}$. For the sake of contradiction, suppose the latter. Then \begin{eqnarray*}
u^4-(3^k+4x^2+1)u^2+3^k&\equiv&0 \pmod{3^{k+1}} \\
\mu^4-(3^k+4(x-a)^2+1)\mu^2+3^k&\equiv&0 \pmod{3^{k+1}}.
\end{eqnarray*}
In the case $x \not\equiv 0 \pmod{3}$, then we saw before $u^2 \equiv 0 \pmod{3^{k}}$. The only possibility is $u^2 \equiv 3^k \pmod{3^{k+1}}$. The first equation becomes $$-(3^k +4x^2+1)3^k + 3^k \equiv 0 \pmod{3^{k+1}}$$ so $$-(4x^2+1)+1\equiv 0 \pmod 3.$$
Thus $x \equiv 0 \pmod{3}$, contrary to the assumption $x \not \equiv 0 \pmod{3}$. We get an analogous contradiction if we assume $x \not \equiv a \pmod{3}$.

This means $z^2 \equiv 0 \pmod{3^{k+1}}$. By induction, $z = 0.$

\begin{figure}[!h]  
\centering 
\begin{tikzpicture}
\draw[help lines, color=gray!30, dashed] (-.5,-1.5) grid (4.9,2.9);
\draw[->, thick] (-.5,0)--(5,0) node[right]{$x$};
\draw[->, thick] (0,-1.5)--(0,3) node[above]{$y$};
\draw (0,.5)--(4.1,0) node[above]{$(x,0)$};
\draw (0,-.5)--(4.1,0); 
\draw (1.65,-.5)--(4.1,0); 
\draw (1.65,.5)--(4.1,0); 
\draw (0,.5)--(0,-.5);
\draw (0,.5)--(1.65,.5);
\draw (1.65,.5)--(1.65,-.5); 
\draw (1.65,-.5)--(0,-.5)  node[midway, below]{$a$};
\end{tikzpicture}
\begin{tikzpicture}
\draw[help lines, color=gray!30, dashed] (-0.5,-1.5) grid (4.9,2.9);
\draw[->, thick] (-.5,0)--(5,0) node[right]{$x$};
\draw[->, thick] (0,-1.5)--(0,3) node[above]{$y$};

\draw (0,.5)--(.303,1.96) node[right]{$(1/(2a),x/a-1/2)$};
\draw (0,-.5)--(.303,1.96); 
\draw (.606,-.5)--(.303,1.96); 
\draw (.606,.5)--(.303,1.96); 
\draw (0,.5)--(0,-.5);
\draw (0,.5)--(.606,.5);
\draw (.606,.5)--(.606,-.5); 
\draw (.606,-.5)--(0,-.5)  node[midway, below]{$1/a$};
\end{tikzpicture}
\caption{ If $(x,0)$ is rational distance from $\Ra$, then $(1/(2a),x/a-1/2)$ is rational distance from $\mathcal{R}_{1/a}$. } \label{fig:shiftrotate}  
\end{figure}
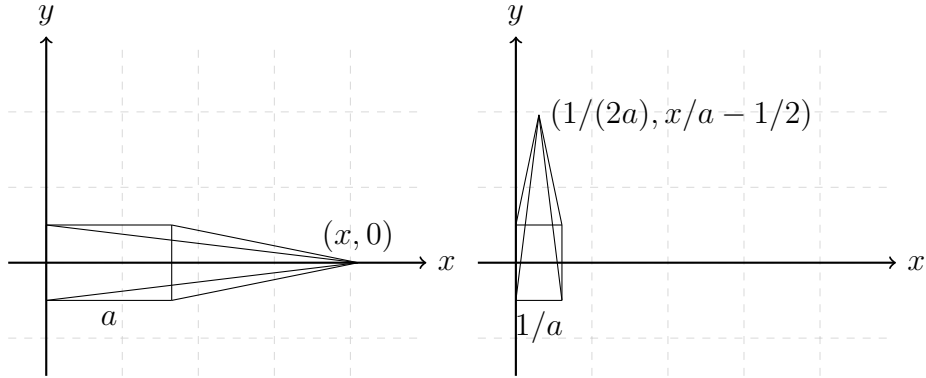

We have shown that all the points rational distance from $\Ra$ are of the form $(x, 0).$ To eliminate the points $(x,0)$, note that then $(1/(2a),x/a-1/2)$ is rational distance from $\mathcal{R}_{1/a}$. (See Figure \ref{fig:shiftrotate}.) Now $v_3(1/a)=0$, $v_3(1/(2a))=0$, and $v_3(x/a-1/2)\geq 0$, so our proof so far shows $x/a-1/2=0$ and thus $x=a/2$. 

We are left with just one possible point rational distance from $\Ra$, which is $(a/2,0)$. This would imply that $a,1$ are legs of a Pythagorean triangle. We can write $a=\frac{r}{s}$ where $r$ and $s$ are coprime and form the legs of a primitive Pythagorean triple. But exactly one of the legs of a primitive Pythagorean triple is divisible by three, which is not possible since $v_3(a)=0.$ Thus we must have had $v_3(x)<0$ or $v_3(z)<0$ from the beginning. \end{proof}

At this point, one can try to generalize the method in Proposition \ref{prop:v3hv3z<0}. One way is to clear powers of 3 out of the denominators of $x$ and $z$. This approach yields partial results.

\begin{prop} \label{prop:v3hneqv3z}
If $(x,z/2)$ is rational distance from $\Ra$, then $v_3(x) \neq v_3(z)$.
\end{prop}

\begin{proof}
In Proposition \ref{prop:v3hv3z<0}, we saw $v_3(x)<0$ or $v_3(z) < 0$. If also $v_3(x)=v_3(z)=k$, then consider the distances $R_1$ and $R_2$, which satisfy 
\[x^2 + (z/2 - 1/2)^2 = R_1^2\]
\[x^2 + (z/2 + 1/2)^2 = R_2^2.\]

Since $k$ is negative, this implies $v_3(R_1) = v_3(R_2) = k$. We can then choose $u$ from $R_1 \pm R_2$ such that $u$ is nonzero and $v_3(u)=k$. By Theorem \ref{phzu} and Corollary \ref{cor:ratdistpoly}, we get a solution with $u$ nonzero to \[u^4-(z^2+4x^2+1)u^2+z^2=0.\]
Write $ \square' = 3^k  \square$ so that $v_3(\square') = 0$ for $ \square=u,x,z$. Then multiplying both sides by $3^{4k}$, we get
\[{u'}^4-({z'}^2+4{x'}^2 + 3^{2k}){u'}^2+3^{2k}{z'}^2 = 0,\] so \[{u'}^4-({z'}^2+4{x'}^2){u'}^2 \equiv 0 \pmod{3}.\]
None of the variables are $0 \pmod{3}$ by assumption, so the lefthand side is just $-1 \pmod{3}$. This proves $v_3(x) \neq v_3(z).$
\end{proof}

We will continue the discussion of generalizing Proposition \ref{prop:v3hv3z<0} after recording our results so far. In sum, we have proved the following theorem.

\begin{thm} \label{thm:congruencecondition}
Suppose that $(x,z/2)$ is rational distance from $\Ra$, where $a$ is rational with $v_3(a) = 0$. Then $v_3(x)<0$ or $v_3(z)<0$, and also $v_3(x) \neq v_3(z).$
\end{thm}

If $(x,z/2)$ is a solution to the two-distance problem, then $(x,z/2)$ is rational distance from $\mathcal{R}_x$ since $z$ is rational. So we get the following result for the two-distance problem.
\begin{cor} \label{cor:v3<0fortriangle}
Suppose that $(x,z/2)$ is a solution to the two-distance problem. If $v_3(x)=0$, then $v_3(z)<0$.
\end{cor}

\begin{rem}
The properties of Theorem \ref{thm:congruencecondition} are specific to the four-distance problem. Using one of the parametrizations for the three-distance problem from Berry's paper \cite{Berry}, we generated the point $(x,z/2) = (6493/28900,12463/14450)$ that is rational distance from $(0, 1/2), (1, \pm 1/2)$, and this point has $v_3(x)=v_3(z)=0$.
\end{rem}

Theorem \ref{thm:congruencecondition} says that it is somehow more difficult to solve the four-distance problem for the rectangle of side lengths $1,a$ if $v_3(a)=0$. Some results of Bremner and Ulas in \cite{BremnerUlas} support this observation. Bremner and Ulas considered the set of rational $a$ with infinitely many points rational distance from $\Ra$. They proved that this set is dense in $\mathbb{R}$ and, using results from Shute and Yocom, they showed that this set remains dense even if the rational distance points are constrained to lie in the interior of the rectangle. In proving this, they remarked on the interesting property that all the $a$ they found were of the form $a = 2t/(1-t^2)$ or $(1-t^2)/(2t)$. They were then interested in finding $a$ not of this form. To do so, they ran a small computer search. After discarding all of the points on the lines $x = 0$, $x=a/2$, $x=a$, $y=-1/2$, $y=0$, $y = 1/2$, the $a$ that remained all satisfied $v_3(a) \neq 0$. 

Notice that also $v_3(2t/(1-t^2))$ is nonzero for any $t$ since $2t/(1-t^2)$ is the ratio of two legs of a Pythagorean triple. The result of Shute and Yocom used by Bremner and Ulas does not help to find a counterexample either. Shute and Yocom's result gives a point in the interior of the rectangle with $a = A/B$, where $p_i, q_i$, and $r_i$ are Pythagorean triples for $i=1,2$ and $A$ and $B$ are defined by $A = p_1q_2+p_2q_1, B = p_1p_2+q_1q_2$. All of these $a$ also satisfy $v_3(a) \neq 0.$ Because of this evidence, we previously posed the following question.

\begin{quest}
Suppose that $v_3(a) = 0$. Is it possible to find a point $(x,y)$ with $x \neq 0,a/2,a$ and $y \neq 0,\pm 1/2$ that is rational distance from $\Ra$? If so, is there an $a$ with infinitely many such points $(x,y)$?
\end{quest}

A small computer search found the solution $(-8/13,-25/78)$ for $a=4/11,$ which answers the first question affirmatively.

\section{Further Discussion}

We now return to generalizing the proof of Proposition \ref{prop:v3hv3z<0}, again assuming $v_3(a) = 0.$ The difficulty is the case $v_3(x) \neq v_3(z).$ By a transformation in the manner of Figure \ref{fig:shiftrotate}, we can narrow the problem down to the case $v_3(z)<v_3(x)$. Let $k=v_3(z)$, so that $k$ is negative. Recall that our system of equations is \begin{align}
u^4-(z^2+4x^2+1)u^2+z^2&=0 \\
\mu^4-(z^2+4(x-a)^2+1)\mu^2+z^2&=0,
\end{align} where we are looking for solutions in rationals with $u,\mu \neq 0$. Just as in Proposition \ref{prop:v3hneqv3z}, we can show that if a point $(x,z/2)$ is rational distance from $\Ra$, then there exists a solution $(x,z,u,\mu)$ to this system of equations with $v_3(u)=v_3(\mu)=k.$ Writing $\square' = 3^k  \square,$ so that $v_3(\square') = 0$ for $ \square=u,\mu,z$, we can multiply both sides by $3^{4k}$ to get \begin{align}
{u'}^4-({z'}^2+3^{2k}(4x^2+1)){u'}^2+3^{2k}{z'}^2&=0 \\
{\mu'}^4-({z'}^2+3^{2k}(4(x-a)^2+1)){\mu'}^2+3^{2k}{z'}^2&=0,
\end{align}

In Proposition \ref{prop:v3hv3z<0}, we succeeded because the solutions (mod $3^k$) were easy to understand in that always $z\equiv0$ (mod $3^k$). That is, the solutions in $3$-adic integers satisfy an extra equation $z=0$, and that allows us to show that there are no rational $3$-adic integer solutions. Ideally, we would like to understand the 3-adic integer solutions to (9) and (10) and similarly determine whether they can be rational. Unfortunately, we have been unsuccessful so far. One possible explanation is that the derivative matrices (mod 3) for the system (7),(8) are not full rank whereas the derivative matrices for the system (9),(10) are full rank. Indeed, the solutions (mod $3^{k+1}$) are obtained from the solutions (mod $3^k$) in the following manner. (The proof is analogous to that of Hensel's lemma.)

\begin{lem} \label{lem:Hensel's type lemma}
Suppose that $f_i \in \mathbb{Z}[X_1,\ldots,X_n].$ The solutions to $f_i(x)\equiv 0 \pmod{3^{k+1}}$ are obtained by setting $x_j=b_j+v_j3^k$, where $f_i(b)\equiv 0 \pmod{3^k}$ and $v$ satisfies \[Df_i(b)v \equiv -f_i(b)/3^k \pmod{3}.\]
\end{lem}

Thus, if the derivative matrix of the polynomial system $f(x)=0$, where $f=(f_1,\ldots,f_m)$, is not full rank$\pmod{3}$, there can be difficulty solving for $v.$ For the system (7),(8), we have a matrix $Df(x_0,z_0,u_0,\mu_0)$ given by

\[  \begin{bmatrix}
    -8x_0u_0^2 & -2z_0u_0^2+2z_0 & 4u_0^3-2u_0(z_0^2+4x_0^2+1) & 0 \\
    -8x_0\mu_0^2 & -2z_0\mu_0^2+2z_0 & 0 & 4\mu_0^3-2\mu_0^2(z_0^2-4(x_0-a)^2+1)
  \end{bmatrix}
\]

We already know from the proof of Proposition \ref{prop:v3hv3z<0} that $z_0 \equiv 0 \pmod{3}$. One can check this implies either $u_0\equiv 0 \pmod{3}$ or $\mu_0 \equiv 0 \pmod{3}$, which means all of the solutions $(x_0,z_0,u_0,\mu_0)$ (mod 3) have $Df(x_0,z_0,u_0,\mu_0)$ not full rank (mod 3). However, for the system (9),(10), the matrix $Df(x_0,{z'}_0,{u'}_0,{\mu'}_0)$ is
\[
  \begin{bmatrix}
    0 & {z'}_0 & -{u'}_0 & 0 \\
    0 & {z'}_0 & 0 & -{\mu'}_0
  \end{bmatrix}
\] which we arrive at using ${z'}_0^2 \equiv {u'}_0^2 \equiv {\mu'}_0^2 \equiv 1 \pmod{3}$. This matrix is full rank for all of the solutions $(x_0,{z'}_0,{u'}_0,{\mu'}_0) \pmod{3}$ because we are restricting ${z'}_0, {u'}_0,{\mu'}_0 \not\equiv 0 \pmod{3}$.

We get somewhat lucky in this respect when we translate our results to the three-and-one system of equations discussed at the beginning of Section 1. 
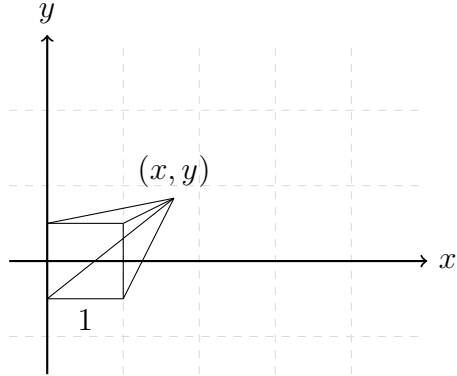
\begin{center}
\begin{tikzpicture}

\draw[help lines, color=gray!30, dashed] (-.5,-1.5) grid (4.9,2.9);
\coordinate (A) at (5/3,4/3-.5);
\draw[->, thick] (-.5,0)--(5,0) node[right]{$x$};
\draw[->, thick] (0,-1.5)--(0,3) node[above]{$y$};
\draw (0,.5)--(A) node[above]{$(x,y)$};
\draw (0,-.5)--(A);
\draw (1,.5)--(A);
\draw (1,-.5)--(A); 
\draw (0,.5)--(0,-.5);
\draw (0,.5)--(1,.5);
\draw (1,.5)--(1,-.5); 
\draw (1,-.5)--(0,-.5)  node[midway, below]{$1$};
\end{tikzpicture}
\end{center}

Indeed, Theorem \ref{thm:congruencecondition} says that if $(x,y)$ solves the four-distance problem, then $v_3(x)<0$ or $v_3(y) < 0$, as well as $v_3(x) \neq v_3(y)$. This implies that $v_3(R_i)= $ min$\{v_3(x),v_3(z)\}$ so $v_3(R_i)<0$, where $R_i$ is the distance of $(x,y)$ to any vertex of the square. Thus, by scaling, we see that in the three-and-one system discussed in Section 1
\bigskip
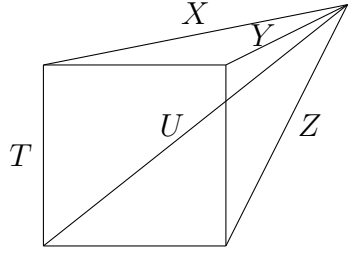
\begin{center}
\begin{tikzpicture}[scale=.8]
\draw (0,0) -- (0,3) node[midway, left]{$T$};
\draw (0,3) -- (3,3);
\draw (3,3) -- (3,0);
\draw (0,0) -- (3,0);
\draw (0,0) -- (5,4) node[midway,left]{$U$};
\draw (3,0) -- (5,4) node[midway,right]{$Z$};
\draw (0,3) -- (5,4) node[midway,above]{$X$};
\draw (3,3) -- (5,4) node[midway,left]{$Y$};
\end{tikzpicture}
\end{center}
\begin{align*}
\tag{1 revisited} 2(Y^4+T^4)+X^4+Z^4&=2(X^2+Z^2)(Y^2+T^2)  \\
\tag{2 revisited} U^2+Y^2&= X^2+Z^2
\end{align*} we have $v_3(X)=v_3(Y)=v_3(Z)=v_3(U)<v_3(T)$.
\begin{prop} \label{prop:guyequationsmod3}
If $X,Y,Z,T,U$ solve the four-distance problem, then $v_3(X)=v_3(Y)=v_3(Z)=v_3(U)<v_3(T)$. In particular, if $X,Y,Z,T,U$ is chosen to be a primitive solution, then $X,Y,Z,U$ are nonzero$\pmod{3}$ and $T$ is zero$\pmod{3}$.
\end{prop} 
The derivative matrix of system (1),(2) is $Df(X_0,Y_0,Z_0,T_0,U_0)$ given by

\[
  \begin{bmatrix}
    a_{11} & a_{12} & a_{13} & a_{14} & 0 \\
    2Y_0 & -2X_0 & -2Z_0 & T_0 & 2U_0
  \end{bmatrix}
\]
where
\[a_{11} = 8Y_0^3-4(X_0^2+Z_0^2)Y_0\]
\[a_{12} = 4X_0^3-4(Y_0^2+T_0^2)X_0\]
\[a_{13} = 4Z_0^3-4(Y_0^2+T_0^2)Z_0\]
\[a_{14} = 8T_0^3-4(X_0^2+Z_0^2)T_0.\]

There are multiple solutions (mod 3) to (1),(2) where this matrix is full rank. However, by Proposition \ref{prop:guyequationsmod3}, we know that any solution to the four-distance problem has $X_0^2\equiv Y_0^2 \equiv Z_0^2 \equiv U_0^2 \equiv 1 \pmod{3}$ and $T_0 \equiv 0 \pmod{3}$. For each of these solutions, the entries in the first row of the matrix are all 0 (mod 3). Perhaps the solutions (mod $3^k$) coming from the four-distance problem can be understood because their derivative matrices are not full rank, just as in the system (7),(8). 

We observed some erratic behavior of the number of solutions (mod $3^k$) to the system (1),(2). Consider a solution $(X_0,Y_0,Z_0,U_0,T_0) \pmod{3^k}$ corresponding to the four-distance problem in the sense of Proposition \ref{prop:guyequationsmod3}. If this solution lifts to a single solution (mod $3^{k+1}$), then by Lemma \ref{lem:Hensel's type lemma} it lifts to a total of 81 solutions (mod $3^{k+1}$), since this is the number of points in the nullspace of the rank 1 matrix $Df(X_0,Y_0,Z_0,U_0,T_0)$ (mod 3). However, when we count number of solutions (mod $3^k$) ($k=1,\ldots, 5)$ corresponding to four-distance problem solutions, we get the somewhat surprising sequence 16, 1296, 34992, 1154736, 31177872 with successive quotients $81,27,33,$ and $27$. We have verified the existence of solutions modulo powers of $3$ up to $3^{100}$, but unfortunately we have been unsuccessful in gaining understanding of these solutions.

Finally, we rule out some parametrizations to the three-distance problem. Berry's paper \cite{Berry} listed some such parametrizations, including a quadratic and a quartic parametrization. Berry proved that up to symmetry there are no more degree 2 and degree 4 parametrizations to the three-distance problem, and he proved that these parametrizations do not solve the four-distance problem. Using our results, we can prove that the octic parametrization listed in his paper does not solve the four-distance problem.

\begin{thm}
The octic parametrization for the three-distance problem
\begin{align*}
    X &= t^8-8t^7+12t^6+24t^5-10t^4-24t^3+12t^2+8t+1 \\
    Y &= 8t^7-16t^6-8t^5-8t^3+16t^2+8t\\
    Z &= t^8 + 12t^6 - 32t^5-10t^4-32t^3+12t^2+1\\
    T &= t^8-4t^6+22t^4-4t^2+1
\end{align*}

contains no solutions to the four-distance problem.
\end{thm}
\begin{proof}
If $v_3(t)<0$, then $v_3(T)=8v_3(t)=v_3(Z)$, but we have shown $v_3(Z) < v_3(T).$ If $v_3(t) \geq 0$, then $v_3(Z) \geq 0$. However, the parametrization implies $T \equiv 1 \pmod{3}$, so $v_3(T) = 0$, which is again contrary to $v_3(Z) < v_3(T).$
\end{proof}
It should be noted that we are unable to use this method to rule out the quadratic, quartic, and sextic parametrizations listed in Berry's paper. And while there is only one quadratic parametrization and only one quartic parametrization, there are many octic and sextic parametrizations that Berry showed how to construct but did not list in his paper. We did not generate these parametrizations, though there are likely more we could rule out with the same method.

\section*{Acknowledgements}
This research was conducted as part of an undergraduate honors thesis at Stanford University. I would like to thank my thesis advisor Professor Vakil for all the time and attention he invested into this project in addition to his indispensable guidance throughout. I would also like to thank Professor Soundararajan and Professor Tsai for multiple helpful comments and conversations. Finally, I would like to thank Stanford University for supporting this research with the UAR Major Grant.

\end{document}